\documentclass[a4paper]{amsart}

\usepackage{lmodern}
\usepackage{stmaryrd}
\SetSymbolFont{stmry}{bold}{U}{stmry}{m}{n}
\usepackage{amsmath}
\usepackage{amsfonts}
\usepackage{amssymb}
\usepackage{amsthm}
\usepackage{aliascnt}
\usepackage{bbm} 
\usepackage{pbox}
\usepackage{booktabs}
\usepackage{arydshln}
\usepackage{pinlabel, caption, subcaption}
\usepackage{esvect}
\usepackage{cals}
\usepackage{verbatim}
\usepackage{xfrac}
\usepackage[backend=biber,style=alphabetic,sorting=nty]{biblatex}
\newcommand{\cF}{\mathcal{F}}

\newcommand{\bF}{\mathbb{F}}

\newcommand{\bZ}{\mathbb{Z}}
\newcommand{\Z}{\mathbb{Z}}

\newcommand{\R}{\mathbb{R}}

\usepackage{booktabs}
\usepackage[dvipsnames,svgnames,x11names]{xcolor}
\usepackage{mathtools,url,graphicx,nicefrac}
\usepackage[colorlinks,citecolor=blue,linkcolor=blue,urlcolor=blue,filecolor=blue]{hyperref}
\usepackage{microtype}
\usepackage[margin=1.25in]{geometry}
\usepackage{pdflscape}

\usepackage{tikz, tikz-cd}
\usetikzlibrary{matrix,arrows,decorations.pathmorphing}
\usetikzlibrary{arrows,decorations.pathmorphing,backgrounds,fit,positioning,shapes.symbols,chains,calc}

\numberwithin{theoremcounter}{section}
\newaliascnt{theoremauto}{theoremcounter}

\newaliascnt{Defauto}{theoremcounter}

\newaliascnt{exampleauto}{theoremcounter}

\newaliascnt{lemmaauto}{theoremcounter}

\newaliascnt{propositionauto}{theoremcounter}

\newaliascnt{corollaryauto}{theoremcounter}

\newaliascnt{remarkauto}{theoremcounter}

\newaliascnt{notationauto}{theoremcounter}

\newaliascnt{claimauto}{theoremcounter}

\newaliascnt{warningauto}{theoremcounter}

\newaliascnt{questionauto}{theoremcounter}

\newaliascnt{discussionauto}{theoremcounter}

\newaliascnt{computationauto}{theoremcounter}

\newaliascnt{conjectureauto}{theoremcounter}

\newaliascnt{convauto}{theoremcounter}

\newtheorem{theorem}[theoremauto]{Theorem}

\newtheorem{lemma}[lemmaauto]{Lemma}
\newtheorem{proposition}[propositionauto]{Proposition}
\newtheorem{corollary}[corollaryauto]{Corollary}
\newtheorem*{corollary*}{Corollary}

\newtheorem{atheorem}{Theorem}

\theoremstyle{definition}
\newtheorem{definition}[Defauto]{Definition}

\newtheorem{example}[exampleauto]{Example}
\theoremstyle{remark}

\newcommand{\HFI}{\mathbb{H}^{\text{FI}}}
\newcommand{\FIi}{\mathbb{H}^{\text{FI}}_{i}}
\newcommand{\FIiC}{\mathbb{H}^{\text{FI}}_{i}(C^{-*}}

\DeclareMathOperator*{\sgn}{sgn}
\DeclareMathOperator{\stab}{stab}

\DeclareMathOperator{\PConf}{PConf}
\DeclareMathOperator{\UConf}{UConf}
\DeclareMathOperator{\Br}{Br}
\DeclareMathOperator{\PBr}{PBr}
\DeclareMathOperator{\RP}{\mathbb{R}P}
\DeclareMathOperator{\HH}{H}
\DeclareMathOperator{\FI}{FI}

\title[Stability Patterns for Spherical and Projective Braid Groups]{Stability Patterns for Spherical and Projective \\ Braid Groups}

\author{Sarah Anderson}
\email{ande1324@purdue.edu}
\address{Department of Mathematics, Purdue University, 150 North University, West Lafayette 47907, United States}

\AtBeginDocument{%
	\def\MR#1{}}

\date{\today}

\begin{document}
	
\begin{abstract} 
McDuff and Segal proved homological stability for unordered configuration spaces of connected manifolds with non-empty boundary. Later, Church proved representation stability for ordered configuration spaces on compact manifolds. These stability patterns extend to surface braid groups and pure surface braid groups, respectively, whenever the associated configuration space is an Eilenberg--Maclane space. However, these results do not apply to braid groups and pure braid groups on $S^2$ and $\RP^2$ since ordered and unordered configuration spaces on these surfaces are not Eilenberg--Maclane spaces. This paper will use results by Cohen--Pakianathan and the theory of $\FI$-homology to extend known stability patterns to these unknown cases. This is a special case of a more general result proven in this paper: the equivariant homology of ordered and unordered configuration spaces exhibit representation stability and stable periodicity respectively.

\end{abstract}

\maketitle


\section{Introduction}
 
The purpose of this paper is to prove stable periodicity for the group homology of the surface braid groups of $S^2$ and $\RP^2$ as well as representation stability for the group homology of the pure surface braid groups of $S^2$ and $\RP^2$. We will begin by reviewing three different types of stability patterns and their current known results for configuration spaces: homological stability, representation stability, and stable periodicity.

\subsection{Stability Patterns for Configuration Spaces}
Let $M$ be a connected manifold with $dim(M) \ge 2$. We define the ordered configuration spaces on $n$ points as follows: $$\PConf_n(M)=\{(x_0,x_1,..,x_{n-1}) \in M^n | x_i \neq x_j \text{ if } i \neq j\}.$$ Let $S_n$ be the symmetric group on $n$ elements. Then we define the unordered configuration space on $n$ points as follows: $$\UConf_n(M) = \PConf_n(M) \slash S_n.$$

The following homological stability theorem was first proven by Arnol'd \cite{MR274462} in 1970 for $M=\R^2$. In 1975, McDuff \cite{MR358766} proved this result for general manifolds without an explicit range. Later, in 1979, Segal \cite{MR533892} generalized Arnol'd's results to general manifolds with an explicit range.
\begin{theorem}[McDuff, Segal]
    Let $M$ be a connected non-compact manifold. Then $$\HH_i(\UConf_n(M)) \cong \HH_i(\UConf_{n+1}(M))$$ for $n \geq 2i$.
\end{theorem}

However, these results no longer hold when the manifold is compact. Fadell--Van Buskirk \cite{MR141128} demonstrated that failure for $S^2$ as seen in the following example. 

\begin{example}[Fadell--Van Buskirk]
    For $n \geq 2$, $$\HH_1(\UConf_n(S^2)) \cong \Z/(2n-2).$$
\end{example}

The above isomorphism shows that the homology of unordered configurations spaces on $S^2$ is dependent on n, thus $\UConf_n(S^2)$ does not have homological stability. However, R. Nagpal \cite[Theorem F]{MR3358218} (and independently, Cantero--Palmer \cite{MR3398727}) proved that unordered configuration spaces on compact manifolds do exhibit a stability pattern known as stable periodicity. Kupers--Miller \cite{MR3556286} improved the range.  

\begin{theorem}[Cantero--Palmer, Nagpal, Kupers--Miller]
    Let $p$ be a prime number and $M$ be a connected, compact, orientable manifold with $dim(M) \ge 2$. If $n \ge 2i$, then  $$H_i(\UConf_n(M), \bF_p) \cong H_i(\UConf_{n+p}(M), \bF_p).$$
\end{theorem}

It is well known that the homology of ordered configuration spaces does not exhibit homological stablility or even periodicity. For example, $$\HH_1(\PConf_n(\R^2))=\Z^{\frac{n(n-1)}{2}}.$$ One can see, though, that this homology does exhibit a predictable pattern. In 2011, Church proved that ordered configuration spaces exhibit yet another type of stability called representation stability \cite{Church_2011}, which implies polynomial growth and multiplicity stability.

The following representation stability results utilize the language of $\FI$-modules, which will be reviewed in Section 2. See \cite{Church_2011,MR3357185,MR3285226, MR3818071,miller2026fihyperhomologyorderedconfigurationspaces}.

\begin{theorem}[Church, Church--Ellenberg--Farb, Church--Ellenberg--Farb--Nagpal, Church--Miller--Nagpal--Reinhold, Miller--Wilson]
Let $M$ be a connected manifold of dimension $\geq 2$. For $i \geq0$, $$\HH^i(\PConf_n(M))$$ is generated in degree $\le 4i$ and presented in degree $\le 4i+1$.
\end{theorem}

\subsection{Stability Patterns for Braid Groups}
The fundamental groups of ordered configuration spaces and unordered configuration spaces are called the pure braid groups and braid groups respectively.

\begin{definition}
    Let $M$ be a connected manifold with $dim(M) \ge 2$. We define the braid groups on $M$ as follows: 
\begin{itemize}
    \item $\PBr_n(M)=\pi_1(\PConf_n(M))$,
    \item $\Br_n(M)=\pi_1(\UConf_n(M))$.
\end{itemize}
\end{definition}

Note that this paper will focus on the case where $dim(M)=2$. This is the most interesting case, since for $dim(M) > 2$, $\pi_1(Conf_n(M)) = (\pi_1(M))^n$. This is because the higher dimentions allow for ``unbraiding''.

For a surface, $\Sigma$, other than $S^2$ or $\RP^2$, $\UConf_n(\Sigma)$ and $\PConf_n(\Sigma)$ are Eilenberg--MacLane spaces so the homology of configuration spaces is isomorphic to the group homology of their associated braid groups. Combining this observation and the previous stability theorems gives us the following stability results for braid groups. 

\begin{corollary}
    Let $\Sigma$ be a connected surface other than $S^2$ or $\RP^2$.  
    \begin{enumerate}
        
        \item Representation Stability \cite{Church_2011}, \cite{miller2026fihyperhomologyorderedconfigurationspaces}: $\HH^i(\PBr_n(\Sigma))$ has generation degree $\le 4i+1$ and presentation degree $\le 4i +2$.
        
        \item Homological Stability \cite{MR274462}, \cite{MR358766}, \cite{MR533892}: Let $\Sigma$ be non-compact. Then $$\HH_i(\Br_n(\Sigma)) \cong \HH_i(\Br_{n+1}(\Sigma))$$ for $n \geq 2i$.
        
        \item Stable Periodicity \cite{MR3398727}, \cite{MR3358218}, \cite{MR3556286}: Let $p$ be a prime. Then for $n \ge 2i,$ $$\HH_i(\Br_n(\Sigma), \bF_p) \cong \HH_i(\Br_{n+p}(\Sigma), \bF_p).$$
    \end{enumerate} 
\end{corollary}

We will now extend these results to braid groups and pure braid groups of $S^2$ and $\RP^2$, the two previously unknown cases.



\begin{atheorem} \label{PBr Rep Stab}
    Let $\Sigma$ be $S^2$ or $\RP^2$. For $i \geq 0$, then $$\HH^i(\PBr_n({\Sigma}))$$ is generated in degree $\le 4i$ and presented in degree $\le 4i+1$.
\end{atheorem}

\begin{corollary}
    
 \label{Br Stab Period}
    Let $p$ be prime, $i \geq 0$, and let $\Sigma$ be $S^2$ or $\RP^2$.  $$\HH_i(\Br_n(\Sigma);\mathbb F_p) \cong \HH_i(\Br_{n+p^j}(\Sigma) ;\mathbb F_p)$$ for $n \ge 8i+1$ and $p^j > 4i$.
\end{corollary}

\begin{corollary} \label{PBr Polynomial}
    Let $\bF$ be a field and let $\Sigma$ be $S^2$ or $\RP^2$. There is a polynomial $f$ of degree at most $2i$ depending on $\Sigma$, $i$, and $\bF$ such that $$\dim \HH_i(\PBr_n(\Sigma); \bF) = f(n)$$ for $i\geq 0$ and $n \geq 8i+1$.
\end{corollary}




\subsection{Stability for Equivariant Homology}

While it is \textbf{not} true that $$\HH_i(\PConf_n(\Sigma)) \cong \HH_i(\PBr_n(\Sigma)),$$ Cohen--Pakianathan established a model for comparing the group homology of $\PBr_n(\Sigma)$ with the equivariant homology of $\PConf_n(\Sigma)$. Considering $S^3$ as the double cover of $SO(3)$, it naturally acts on $S^2$ by rotation. Let $\sslash$ denote homotopy quotient, also known as the Borel construction; Cohen--Pakianathan \cite[Cor 6.16 and 7.15]{cohen2026configurationspacesbraidgroups} proved the following result.

\begin{theorem}[Cohen--Pakianathan] \label{Jon and Fred}  Let $\Sigma$ be $S^2$ or $\RP^2$, then 
    $$\HH^i(\PConf_n(\Sigma) \sslash S^3) \cong \HH^i(\PBr_n(\Sigma)) \text{ for } n \ge 3$$ and
    $$\HH^i(\UConf_n(\Sigma) \sslash S^3) \cong \HH^i(\Br_n(\Sigma)) \text{ for } n \ge 3.$$

\end{theorem}

Thus the group homology of the braid groups on $\Sigma$ are not the same as the homology of their associated configuration spaces, but they are the same as the equivariant homology of their associated configuration spaces. We will prove the follow stability result for the equivariant homology of configuration spaces.

\begin{atheorem} \label{Equivariant Rep Stab}
    Let $G$ be a group acting on a manifold $M$. For $i \geq 0$, $$\HH^i(\PConf_n(M) \sslash G)$$ is generated in degree $\le 4i$ and presented in degree $\le 4i+1$.
\end{atheorem}

\autoref{Jon and Fred} combined with \autoref{Equivariant Rep Stab} gives us \autoref{PBr Rep Stab}; \autoref{PBr Rep Stab} gives us \autoref{Br Stab Period} and \autoref{PBr Polynomial}. The proofs rely heavily on the theory of $\FI$-Hyperhomology which will be reviewed in Section 2.



\subsection{Acknowledgments} I want to thank my advisor, Jeremy Miller. I would not have been able to complete this paper without his support and guidance. I also want to thank Jennifer Wilson and Nicolas Guès for helpful conversations. Thank you to Jonathan Pakianathan for taking the time to post his notes with Fred Cohen on the Arxiv. I also received financial support from NSF grants DMS-2504473 and DMS-2202943 as well as a Simons Foundation Travel Support for Mathematicians grant.


\section{Review of FI-modules}
This section will give an overview of $\FI$-modules, $\FI$-chain complexes and $\FI$-hyper homology as well as review relevant theorems and lemmas. Throughout this section, $S$ will be used to refer to a finite set and $p \ge 0.$

\begin{definition} \text{ }
    \begin{itemize} 
        \item $\FI$ is the category whose objects are finite sets and whose morphisms are injections.
        \item An $\FI$-module is a covariant functor, $N: \FI \longrightarrow \text{Ab}$, the category of abelian groups. $N_n = N([n])$ where $[n]=\{1,...,n\}$. The category of $\FI$-modules is called Mod$_{\FI}$.
        \item A co-$\FI$-module is a contravariant functor from $\FI$ to abelian groups.
    \end{itemize}
\end{definition}

The above definitions still make sense when the category of abelian groups is replaced with any other category. For example, the term $\FI$-chain complex refers to a covariant functor from $\FI$ to chain complexes.

\begin{definition}
    Let $$\sgn\nolimits_S = \bigwedge\nolimits^{|S|} \bZ[S]$$ viewed as a representation of the group of bijections of $S$.
\end{definition}

\begin{definition} \label{chain FI definition}
    We define the chains on an $\FI$-module, $N$, as follows: $$C_p^{\FI}(N)_S = \underset{\underset{|T|=p}{T \subseteq S}}{\bigoplus} N_{S-T} \otimes \sgn\nolimits_T.$$ If $X=Y \cup \{t\}$, let $$\stab_t: N_Y \longrightarrow N_X$$ be the map induced by $Y \hookrightarrow X.$ Let the map $$\delta_p: C_p^{\FI}(N) \longrightarrow C_{p-1}^{\FI}(N)$$ be given by $$x \otimes t_1 \wedge \cdots \wedge t_p \longmapsto \sum_{i=1}^{p} (-1)^{i}t_i(x) \otimes t_1 \wedge \cdots \hat{t_i} \cdots \wedge t_n$$ where $x \in X$ and $T=\{t_1, \dots, t_p\}$. 
\end{definition}

Church--Ellenberg  \cite[Definition 5.9]{MR3654111} proved that the above map, $\delta_p$, is the differential for a chain complex.

\begin{definition}
    Let $N$ be an $\FI$-module and $P_*$ be an $\FI$-chain complex. We define $\FI$ homology as follows: $$\FIi(N)_S = H_i(C_{*}^{\FI}(N)_{S})$$ where the differentials are $\delta_p$ defined in \autoref{chain FI definition}. \\
    Similarly, we define $\FI$-hyper homology as follows:
    $$\FIi(P_*)_S = \HH_i(C_*)$$ where $C_*$ is the total complex of the double complex $C_*^{\FI}(P_*)_S$ and where the differentials are $\delta_p$ defined in \autoref{chain FI definition} along with the chain maps on $P_*$.
    \end{definition}

\begin{definition}[Church--Ellenberg, \cite{MR3654111}]
Let $N$ be an $\FI$-module.
    \begin{itemize}
        \item $N$ is said to be generated in degree $\le d$, if $\HH^{\FI}_0(N)_S = 0$ for $n > d$.
        \item $N$ is said to be presented in degree $\le d$, if $\HH^{\FI}_0(N)_S = \HH^{\FI}_1(N)_S = 0$ for $n > d.$
    \end{itemize}
\end{definition}

The following proposition is well known and follows from the definition of $\FI$-hyper homology.
\begin{proposition}
If we have a short exact sequence of $\FI$-chain complexes: $$0 \longrightarrow A_*^{\FI} \longrightarrow B_*^{\FI} \longrightarrow C_*^{\FI} \longrightarrow 0,$$ there is an induced long exact sequence on $\FI$-hyper homology: \begin{align*}
    \hspace{3cm}& & & \hspace{1.6cm} \dots & \longrightarrow \HFI_{i+1}(C_*^{\FI})_n & \hspace{5cm}\\ \\
    \hspace{3cm}&\longrightarrow \FIi(A_*^{\FI})_n && \longrightarrow \FIi(B_*^{\FI})_n & \longrightarrow \FIi(C_*^{\FI})_n & \hspace{5cm} \\ \\
    \hspace{3cm}& \longrightarrow \HFI_{i-1}(A_*^{\FI})_n && \longrightarrow \hspace{.1cm} \dots \text{ } & 
\end{align*}
\end{proposition}



The next theorem shows that theory of $\FI$-hyper homology can be used to prove representation stability for the homology of $\FI$-chain complexes.

\begin{theorem} \label{Gues generation and presentation} \cite[Theorem 2.2 (1)]{MR5064143}
Let $X_{*}$ be an $\FI$-chain complex. Suppose that $\FIi (X_{*})_S = 0$ for all $n > -ai+b$ for all i and with $a > 0$. Then every $H_i (X_{*})$ is  
generated in degrees  $\le \mathrm{max}(0, 2ai+2b)$ and presented in degrees $\le \mathrm{max}(0, 2ai+2b+1)$. 

\end{theorem}

\section{Lemmas for co-FI Spaces}
This section will provide proofs of general lemmas about the $\FI$-hyper homology of semi-simplicial co-$\FI$ spaces and homotopy quotients.

\begin{lemma} \label{lemma 1}
Let $W_\bullet$ be a co-$\FI$ semi-simplicial space and assume $\FIi(C^{-*}W_p)=0$ for $n>-2i$, then $\FIi(C^{-*}||W_\bullet||)=0$ for $n>-2i$.
\end{lemma}
\begin{proof} (Note: this proof follows the proof structure of \cite[Proposition 4.7] {miller2026fihyperhomologyorderedconfigurationspaces}.)
Let $C^{*,*}$ be the cohomologically graded double complex with $C^{p,q}=C^{-q}(W_{-p})$ and let $C_*$ be the total complex of $C^{*,*}$. Bendersky--Gitler \cite[Proposition 1.2]{MR1010881} constructed a natural equivalence (see also Miller--Wilson \cite[Proposition 4.6]{miller2026fihyperhomologyorderedconfigurationspaces})  $$C_* \overset{\simeq}{\longrightarrow} C^{-*}(||W_\bullet||)$$ which induces $$H_i(C_*) \cong H^i(||W_\bullet||).$$ 
It is enough to show $$\FIi(C_*)_n = 0 \text{ for } n > -2i.$$ \\
Consider the filtration $$\cF_{p,q}^k = \begin{cases}
                    C_{p,q} & p \le k \\
                    0 & else.
                \end{cases}$$
Notice that $\cF_{*}^{0}=C_*$ and $\cF_{*}^{k} \subseteq \cF_{*}^{k+1}$. \\
Note that the map $$C_j \longrightarrow C_j \slash \cF_{j}^{k}$$ is an isomorphism for $j \ge k+1$. Thus, $$H_j(C_*) \longrightarrow H_j(C_j \slash \cF_{j}^{k})$$ is an isomorphism for $j \ge k+2$. So, for $n>-2i$ and $i \ge 2k+3$, \cite[\color{red}Lemma 4.5]{miller2026fihyperhomologyorderedconfigurationspaces} tells us that the map $$\FIi(C_*)_n \longrightarrow \FIi(C_* \slash \cF_*^k)_n$$ is an isomorphism. Recall $C^{-*}(||W_\bullet||)$ is equivalent to $C_*$, so $\FIi(C_*)_n=\FIi(C^{-*}(||W_\bullet||))_n$. \\ 

So the problem reduces to showing that $$\FIi(C_* \slash \cF_*^k)_n \cong 0 \text{ for } n>-2i$$ for all $k$. We will do this using induction on $k$. \\

\noindent \textbf{Base Case:} If $k \ge 0$, $C_* \slash \cF_*^k=0$, so $\FIi(C_* \slash \cF_*^k)_n \cong 0$. So fix $k<0$. \\
\textbf{Induction Hypothesis}: Assume $\FIi(C_* \slash \cF_*^{k+1}) \cong 0$ and we will show the result for $k$ (since $k<0$, we are counting towards $-\infty$). \\
\textbf{Induction Step}: Consider the following short exact sequence: $$0 \longrightarrow \cF_*^{k+1} \slash \cF_*^k \longrightarrow C_* \slash \cF_*^k \longrightarrow C_* \slash \cF_*^{k+1} \longrightarrow 0$$ which induces the following long exact sequence on $\FI$-hyper homology: 
\begin{align*}
    \hspace{2cm}& & & \hspace{2cm} \dots & \longrightarrow \HFI_{i+1}(C_* \slash \cF_*^{k+1})_n & \hspace{5cm}\\ \\
    \hspace{2cm}&\longrightarrow \FIi(\cF_*^{k+1} \slash \cF_*^k)_n && \longrightarrow \FIi(C_* \slash \cF_*^k)_n & \longrightarrow \FIi(C_* \slash \cF_*^{k+1})_n & \hspace{5cm} \\ \\
    \hspace{2cm}& \longrightarrow \HFI_{i-1}(\cF_*^{k+1} \slash \cF_*^k)_n && \longrightarrow \dots \text{ } &    
\end{align*}

\noindent Notice that $$\cF_*^{k+1} \slash \cF_*^k \cong C^{-q-k-1}(W_{k+1}).$$ Also, by assumption, $$\FIi(C^{-*}(W_k))_n = 0$$ for all $k$ and $n>-2i$, thus $$\FIi(\cF_*^{k+1} \slash \cF_*^k)_n \cong \FIi(C^{-*-k-1}(W_{k+1}))_n \cong \HFI_{i+k+1}(C^{-*}(W_{k+1}))_n=0$$ for $n>-2(i+k+1)$. Therefore, $$\FIi(C_* \slash F_*^k)_n \cong 0$$ for $n>-2i$.
\end{proof}

The following is a corollary to \autoref{lemma 1}.

\begin{corollary} \label{lemma 2}
    Let X be a co-$\FI$ space and Y be a space. If $\FIi(C^{-*}X)_n = 0 \text { for } n>-2i$, then $\FIi(C^{-*}(X \times Y))_n = 0 \text{ for } n>-2i$.
\end{corollary}

\begin{proof}
 Let $Y_\bullet$ be a semi-simplical set such that $Y \underset{w.e.}{\simeq} \vert Y_\bullet \vert$. 
 
 In order to apply \autoref{lemma 1}, we will let $W_\bullet = X \times Y_\bullet$ and check that $\FIi(C^{-*}W_p)=0$ for $n>-2i$.
    
    Notice that since $Y_p$ is a discrete set, $W_p= X \times Y_p =\underset{\alpha \in Y_p}{\bigsqcup} X$. So, $$\FIi(C^{-*}(W_p)_n=\FIiC(\underset{\alpha \in Y_p}{\bigsqcup} X))_n=H_i(C^{\FI}_{*}(C^{-*}(\underset{\alpha \in Y_p}{\bigsqcup} X))_n)=H_i(C^{\FI}_{*}(\underset{\alpha \in Y_p}{\bigoplus} C^{-*}(X))_n)$$ \hspace{0.5cm} Since $C^{\FI}_*$ is made up of finite direct sums, and finite direct sums commute with products, $$H_i(C^{FI}_{*}(\underset{\alpha \in Y_p}{\bigoplus} C^{-*}(X))_n)=H_i(\underset{\alpha \in Y_p}{\Pi} (C^{\FI}_{*}(C^{-*}(X))_n))$$ \hspace{0.5cm} Since homology commutes with products (see Weibel  \cite[Exercise 1.2.1]{MR1269324}), $$H_i(\underset{\alpha \in Y_p}{\Pi} (C^{\FI}_{*}(C^{-*}(X))_n))=\underset{\alpha \in Y_p}{\Pi} (H_i(C^{\FI}_{*}(C^{-*}(X))_n))=\underset{\alpha \in Y_p}{\Pi} (\FIi((C^{-*}(X)))_n).$$ Now $$\FIi(C^{-*}(X))_n=0 \text{ for } n > -2i$$ implies $$\underset{\alpha \in Y_p}{\Pi} (\FIi((C^{-*}(X)))_n)=0 \text{ for } n > -2i.$$
\end{proof}

\begin{theorem} \label{equivariant chain complex}
    Let $G$ be a group acting on a manifold, $M$, then $\FIi(C^{-*}(\PConf_n(M) \sslash G))=0$ for $n > -2i$.
\end{theorem}

\begin{proof}
    Notice that $\PConf_n(M) \sslash G$ is the geometric realization of the two sided bar construction: $$\big|\big|B_\bullet \big( \{pt\}, G, \PConf_n(M)\big)\big|\big|.$$ Let $W_p \coloneq G^p \times \PConf_n(M)$ $\left(\text{note that } G^p=B_p\big(\{pt\}, G, \PConf_n(M)\big) \right)$. If it can be shown that $\FIi\big(C^{-*}(W_p)\big)=0$ for $n>-2i$, then applying \autoref{lemma 1} will give the desired outcome. \\
    Miller-Wilson, \cite[Proposition 4.7]{miller2026fihyperhomologyorderedconfigurationspaces}, proved that $$\FIi\big(C^{-*}(\PConf_n(M)\big)=0 \text{ for } n>-2i.$$ Since $\PConf_n(M)$ is a co-$\FI$ space and $G^p$ is a space for all $p$, applying \autoref{lemma 2} to gives $$\FIi\big(C^{-*}(G^p \times \PConf_n(M)\big)=0$$ for all $p$ when $n>-2i$.
\end{proof}

\section{Applications to Configuration Spaces}

This section will provide proofs of the main theorems of this paper. The strategy will be to first focus on the results pertaining to ordered configuration spaces and pure braid groups. We will then use those theorems, along with the work of Nagpal and Guès, to determine the stability patterns for spherical braid groups and the equivariant homology of unordered configuration spaces.

\subsection{Ordered configuration spaces and pure braid groups}
In this section, we will prove \autoref{PBr Rep Stab}, \autoref{Equivariant Rep Stab}, and \autoref{PBr Polynomial}.
The following is identical to \autoref{Equivariant Rep Stab} and is being restated for convenience.

\begin{theorem}
Let $G$ be a group acting on a manifold $M$. For $i \geq 0$, $$\HH^i(\PConf_n(M) \sslash G)$$ is generated in degree $\le 4i$ and presented in degree $\le 4i+1$.
\end{theorem}

\begin{proof}
    Let $X_* = C^{-*}(\PConf(M) \sslash G)$. \autoref{equivariant chain complex} tells us that $C^{-*}(\PConf(M) \sslash G)$ satisfies the assumptions of \autoref{Gues generation and presentation} for a=2 and b =0. That gives us the stated generation and presentation degree for $\HH^i(\PConf(M) \sslash G)$.
\end{proof}

Since the equivariant homology of ordered configuration spaces exhibits representation stability, \autoref{Jon and Fred} can be used to complete the proof of \autoref{PBr Rep Stab}, which is restated below.

\begin{theorem} 
    Let $\Sigma$ be $S^2$ or $\RP^2$. For $i \geq 0$, then $$\HH^i(\PBr_n({\Sigma}))$$ is generated in degree $\le 4i$ and presented in degree $\le 4i+1$.
\end{theorem}

\begin{proof}
    Let $\Sigma=S^2$ or $\RP^2$ and let $G = S^3$. By \autoref{Jon and Fred}, $$\HH^i(\PConf_n(\Sigma) \sslash S^3) \cong \HH^i(\PBr_n(\Sigma))$$ for $n \ge 3$. \autoref{Equivariant Rep Stab} gives us the desired generation and presentation degree.
\end{proof}

 \autoref{PBr Polynomial} follows immediately from \autoref{PBr Rep Stab} and \cite[Proposition 3.1 and Proposition 2.14]{MR3818071}.

\subsection{Unordered configuration spaces and braid groups}
In order to give a proof of  \autoref{Br Stab Period}, it is necessary to review the following result of Guès 
\cite[Proposition 3.13]{gues2025periodicityhomologymodulispaces} 
, which is an improvement on early work by Nagpal \cite{MR3358218}.

\begin{proposition}[Guès] \label{Nic Prop 3.13}
    Let X be a co-$\FI$ space and let $S_n$ act freely on $X_n$, and suppose \begin{align*}
        & \HH^i(X) \text{ is generated in } deg\le ai+b \text{ and} \\
        & \HH^i(X) \text{ is presented in } deg\le ai+b +1.
    \end{align*} Then $$\HH^i(X_n \sslash S_n; \bF_p)=\HH^i(X_{n+p^j} \sslash S_{n+p^j}; \bF_p),$$ where $p^j>ai+b$ for a prime $p$ and $n\ge 2ai+2b+1$ if $a>2$.
\end{proposition}

\begin{corollary} \label{unordered conf equivariant periodicity}
    $\HH^i(\UConf_n(M) \sslash G ; \bF_p) = \HH^i(\UConf_{n+p^j}(M) \sslash G; \bF_p)$ where $p^j > 4i-1$ and $n \ge 8i +1.$
\end{corollary}

\begin{proof}
    Recall that $\UConf_n(M)=\PConf_n(M) \slash S_n$ for any manifold, $M$. Also note that since the action of $S_n$ on  $\PConf_n(M)$ is free, $$\PConf_n(M) \slash S_n \simeq \PConf_n(M) \sslash S_n.$$ Also, recall that the action of $S_n$ commutes with the action of $G$. So, since homotopy orbits of commuting group actions commute with each other, $$\UConf_n(M) \sslash G \simeq (\PConf_n(M) \sslash S_n) \sslash G \simeq (\PConf_n(M) \sslash G) \sslash S_n \simeq(\PConf_n(M) \sslash G) \sslash S_n.$$ Taking $X_n = \PConf_n(M) \sslash G$  and applying \autoref{Nic Prop 3.13} completes the proof.
\end{proof}

 \autoref{Br Stab Period} follows from applying \autoref{Jon and Fred} to \autoref{unordered conf equivariant periodicity}.

\printbibliography
\end{document}